\documentclass[12pt,reqno]{amsart}

\pdfoutput=1

\usepackage{amssymb}

\usepackage{tikz}
\usepackage{faktor}
\usepackage{float}

\usepackage{hyperref}
\usepackage{mathtools}
\mathtoolsset{showonlyrefs}

\newtheorem{thm}{Theorem}[section]

\newtheorem{lem}[thm]{Lemma}
\newtheorem{cor}[thm]{Corollary}

\theoremstyle{definition}
\newtheorem{definition}[thm]{Definition}

\newtheorem{notn}[thm]{Notation}

\theoremstyle{remark}

\renewcommand{\epsilon}{\varepsilon}

\newcommand{\bb}[1]{\mathbb{#1}}

\newcommand{\R}[4]{R_{#1}(#2 \to #3 #4)}
\newcommand{\Rstar}[4]{R_{#1}^*(#2 \to #3 #4)}

\makeatletter
\DeclareRobustCommand\widecheck[1]{^{{\mathpalette\@widecheck{#1}}}}
\def\@widecheck#1#2{%
    \setbox\z@\hbox{\m@th$#1#2$}%
    \setbox\tw@\hbox{\m@th$#1%
       \widehat{%
          \vrule\@width\z@\@height\ht\z@
          \vrule\@height\z@\@width\wd\z@}$}%
    \dp\tw@-\ht\z@
    \@tempdima\ht\z@ \advance\@tempdima2\ht\tw@ \divide\@tempdima\thr@@
    \setbox\tw@\hbox{%
       \raise\@tempdima\hbox{\scalebox{1}[-1]{\lower\@tempdima\box
\tw@}}}%
    {\ooalign{\box\tw@ \cr \box\z@}}}
\makeatother

\begin{document}

\title[A decoupling proof of the Tomas restriction theorem]{A decoupling proof of the Tomas restriction theorem}
\author{Griffin Pinney}
\address{Mathematical Sciences Institute, Australian National University, Canberra, Australia}
\email{\href{mailto:griffin.pinney@anu.edu.au}{griffin.pinney@anu.edu.au}; \href{mailto:griffinpinney1@gmail.com}{griffinpinney1@gmail.com}}

\keywords{Fourier restriction, Fourier extension, decoupling}
\thanks{I would like to thank my advisor Po-Lam Yung for initiating my interest in the Fourier restriction problem, and for guiding me throughout my research.}

\begin{abstract}
We give a new proof of a classic Fourier restriction theorem for the truncated paraboloid in $\bb{R}^n$ based on the $l^2$ decoupling theorem of Bourgain-Demeter. Focusing on the extension formulation of the restriction problem (dual to the original restriction formulation), we find that the $l^2$ decoupling theorem directly implies a local variant of the desired extension estimate incurring an $\epsilon$-loss. To upgrade this result to the desired global extension estimate, we employ some $\epsilon$-removal techniques first introduced by Tao. By adhering to the extension formulation, we obtain a more natural proof of the required $\epsilon$-removal result.
\end{abstract}
\maketitle

\section{Introduction}

Fix $n\geq 2$ and define the \textit{truncated paraboloid} in $\bb{R}^n$ by
$$P^{n-1}:=\{(\xi,|\xi|^2)\in\bb{R}^n:\;\xi\in[-1,1]^{n-1}\},$$
which we equip with the surface measure $d\sigma$. We define the \textit{extension operator} $E$ on $L^p(P^{n-1},d\sigma)$ by
$$Ef(x) = \int_{P^{n-1}}f(\xi)e^{2\pi ix\cdot\xi}\,d\sigma(\xi),$$
and denote by $\Rstar{}{p}{q}{}$ the statement that $E$ is a bounded operator $L^p(P^{n-1},d\sigma)\to L^q(\bb{R}^n)$. Thus, $\Rstar{}{p}{q}{}$ is equivalent to the \textit{extension estimate}
\begin{equation}
    \|Ef\|_{L^q(\bb{R}^n)} \lesssim \|f\|_{L^p(P^{n-1},\,d\sigma)},\quad\forall f\in L^p(P^{n-1},d\sigma). \label{eq:extensionest}
\end{equation}
The notation $\Rstar{}{p}{q}{}$ is motivated by the equivalence of estimate \eqref{eq:extensionest}, by a standard duality argument, to the corresponding \textit{restriction estimate}
\begin{equation}
    \|\hat{g}|_{P^{n-1}}\|_{L^{p'}(P^{n-1},\,d\sigma)} \lesssim \|g\|_{L^{q'}(\bb{R}^n)},\quad\forall g\in\mathcal{S}(\bb{R}^n), \label{eq:restrictionest}
\end{equation}
which we denote by $\R{}{q'}{p'}{}$. The restriction problem is rooted in estimates of the form \eqref{eq:restrictionest}, and much of the terminology we will employ is therefore phrased in terms of restriction. Despite this, we will continue to focus on extension estimates.

The \textit{restriction conjecture} states that $\Rstar{}{p}{q}{}$ holds for all exponents $p$ and $q$ satisfying
$$\frac{1}{q} < \frac{n-1}{2n},\quad\frac{n+1}{q} \leq \frac{n-1}{p'}.$$
Thus, the conjectured strong-type diagram for $E$ is as follows:

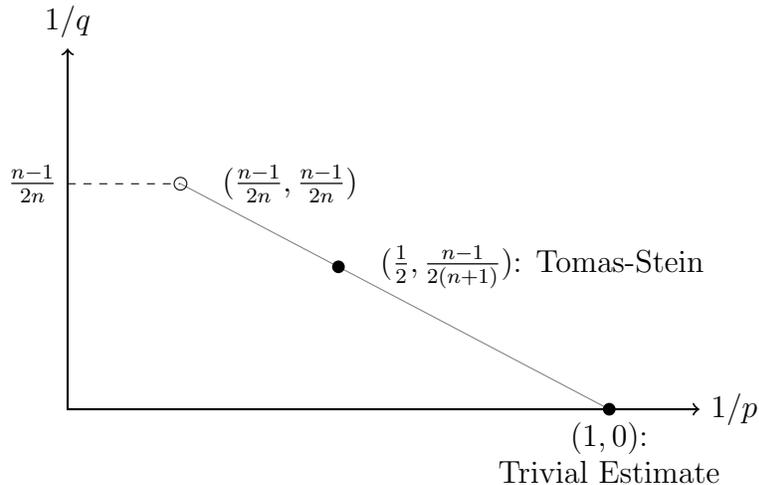
\begin{figure}[H] \label{fig:conjecture}
\begin{tikzpicture}[scale=6]
    \draw [<->,thick] (0,0.8) node (yaxis) [above] {$1/q$}
        |- (1.4,0) node (xaxis) [right] {$1/p$};
    \draw[dashed] (0,1/2) -- (1/4,1/2);
    \draw[gray] (1.2,0) -- (1/4,1/2);
    \fill[black] (1.2,0) node[anchor=north] {\begin{tabular}{c}$(1,0)$: \\ Trivial Estimate \end{tabular}} circle (0.4pt);
    \fill[black] (0.6,6/19) node[anchor=west] {\quad$(\frac{1}{2},\frac{n-1}{2(n+1)})$: Tomas-Stein} circle (0.4pt);
    \draw (1/4,1/2) node[anchor=west] {\quad$(\frac{n-1}{2n},\frac{n-1}{2n})$} circle (0.4pt);
    \draw(0,1/2) node[anchor=east] {$\frac{n-1}{2n}$};
\end{tikzpicture}
\caption{The conjectured strong-type diagram for the extension operator.}
\centering
\end{figure}

Early progress on the restriction conjecture was made by Tomas \cite{Tomas1}, using interpolation methods to prove $\Rstar{}{2}{q}{}$ for all $q>\frac{2(n+1)}{n-1}$, a result which we will refer to as the \textit{Tomas restriction theorem}. This result of Tomas was improved by Stein in the same year to include the endpoint $q=\frac{2(n+1)}{n-1}$, which we therefore refer to as the \textit{Tomas-Stein exponent}. This improvement of Stein was first published by Tomas \cite{Tomas2}, and a proof may also be found in Chapter IX, Section $2.1$ of \cite{Stein}. Further developments, proving restriction theorems of Tomas-Stein type for more general measures, have since been made in \cite{Mockenhaupt,Mitsis,BS}; recent progress on the restriction conjecture itself has been made in \cite{Guth1, Guth2, HR, Wang, HZ}. We will present a new proof of the Tomas restriction theorem based on the recent \textit{$l^2$ decoupling theorem} of Bourgain-Demeter \cite{BD}. We note that the idea of using decoupling to prove restriction results is not entirely new; an application of decoupling to a discrete restriction theorem is also given in \cite{BD}.

Fix $C\geq 1$. For each $\xi\in\bb{R}^{n-1}$ and each $\delta>0$, let $P_{\xi,\delta}\subset\bb{R}^n$ denote the region
$$P_{\xi,\delta} = \{(\xi',\xi_n) : \xi'\in \xi+(-\delta,\delta)^{n-1}\,;\,|\xi_n-|\xi|^2-2\xi\cdot(\xi'-\xi)| < C\delta^2\}.$$
This can be thought of as the hyperplane tangent to the paraboloid at $(\xi,|\xi|^2)$, thickened to scale $\sim\delta^2$ in the $\xi_n$ direction and lying above a cube of scale $\sim\delta$ about $\xi$. The $l^2$ decoupling theorem of Bourgain and Demeter may be reformulated as follows:

\begin{thm}[The $l^2$ decoupling theorem: Theorem $1.1$, \cite{BD}]\label{thm:decforparaboloid}
Let $0<\delta\leq 1$, and let $\Sigma\subset[-1,1]^{n-1}$ be $\delta$-separated. If for each $\xi\in\Sigma$, $f_{\xi}\in\mathcal{S}(\bb{R}^n)$ has Fourier support in $P_{\xi,\delta}$, then
$$\Big\|\sum_{\xi\in\Sigma}f_{\xi}\Big\|_{L^{\frac{2(n+1)}{n-1}}(\bb{R}^n)} \lesssim_{\epsilon} \delta^{-\epsilon}\Big(\sum_{\xi\in\Sigma}\|f_{\xi}\|_{L^{\frac{2(n+1)}{n-1}}(\bb{R}^n)}^2\Big)^{1/2}$$
for all $\epsilon>0$.
\end{thm}

Note that the decoupling exponent $\frac{2(n+1)}{n-1}$ afforded to us by Theorem \ref{thm:decforparaboloid} is the same as the Tomas-Stein exponent. This is the primary fact we will exploit to give a new proof of the Tomas restriction theorem based on decoupling.

Inherent to the decoupling theorem is the $\epsilon$-loss in the factor $\delta^{-\epsilon}$. Consequently, we may only hope to to use the decoupling theorem to directly prove a version of the extension estimate $\Rstar{}{p}{q}{}$ which also incurs an $\epsilon$-loss. We therefore introduce a ``local" variant of extension estimates.

\begin{definition}
Let $\epsilon>0$. If we have
\begin{equation}
    \|Ef\|_{L^{q}(B(0,R))} \lesssim R^{\epsilon}\|f\|_{L^{p}(P^{n-1},\,d\sigma)} \label{eq:R_S^*(p to q,epsilon)}
\end{equation}
for all $R\geq 1$ and all $f\in L^{p}(P^{n-1},d\sigma)$, we say that $\Rstar{}{p}{q}{\,;\,\epsilon}$ holds.
\end{definition}

There is a general principle known as \textit{$\epsilon$-removal} which states that the local extension estimates $\Rstar{}{p}{q}{\,;\,\epsilon}$ for all $\epsilon>0$ should imply a global extension estimate in which the $\epsilon$-loss is transferred to one or both of the exponents $p,q$. This idea was first explored by Tao \cite{TaoBR}, and an adaptation of Tao's methods yields the following concrete result to be explored in the final section:

\begin{thm}[$\epsilon$-removal] \label{thm:epsilonremoval}
The local extension estimates $\Rstar{}{2}{\frac{2(n+1)}{n-1}}{\,;\,\epsilon}$ for all $\epsilon>0$ imply $\Rstar{}{2}{q}{}$ for all $q>\frac{2(n+1)}{n-1}$.
\end{thm}

Our goal is therefore to use the $l^2$ decoupling theorem to prove the local extension estimates $\Rstar{}{2}{\frac{2(n+1)}{n-1}}{\,;\,\epsilon}$ for all $\epsilon>0$. Do to so, it is helpful to reformulate the local extension estimates, so we will use the following well-known equivalent condition for $\Rstar{}{p}{q}{\,;\,\epsilon}$. In what follows, we let $\mathcal{N}_{\delta}$ denote the $\delta$-neighbourhood of $P^{n-1}$ for any $\delta>0$.

\begin{lem} \label{lem:equivcondition}
Let $1\leq p,q\leq\infty$ and $\epsilon>0$. Then, $\Rstar{}{p}{q}{\,;\,\epsilon}$ holds if and only if for all $R\geq 1$ and all $g\in\mathcal{S}(\bb{R}^n)$ with Fourier support in $\mathcal{N}_{R^{-1}}$, we have $\|g\|_{L^q(B(0,R))}\lesssim R^{\epsilon-1/p'}\|\hat{g}\|_{L^p(\bb{R}^n)}$.
\end{lem}

\section{A decoupling proof of the Tomas restriction theorem}

In light of Theorem \ref{thm:epsilonremoval} and Lemma \ref{lem:equivcondition}, to prove the Tomas Restriction theorem, it is sufficient to prove that for all $R\geq 1$ and all $g\in\mathcal{S}(\bb{R}^n)$ with Fourier support in $\mathcal{N}_{R^{-1}}$, we have
\begin{equation}
    \|g\|_{L^{\frac{2(n+1)}{n-1}}(B(0,R))} \lesssim_{\epsilon} R^{\epsilon-1/2}\|\hat{g}\|_{L^2(\bb{R}^n)} \label{eq:suff}
\end{equation}
for all $\epsilon>0$. In addition to the decoupling theorem, the main ingredient will be a variant of \textit{Bernstein's inequality} specific to the regions $P_{\xi,\delta}$.

\begin{thm} \label{thm:Bernstein}
Let $\xi\in\bb{R}^{n-1}$ and $\delta>0$, and suppose $g\in\mathcal{S}(\bb{R}^n)$ has Fourier support in $P_{\xi,\delta}$. Then, for all $1\leq p\leq q\leq\infty$, we have
\begin{equation}
    \|g\|_{L^q(\bb{R}^n)} \lesssim_{p,q} |P_{\xi,\delta}|^{1/p-1/q}\|g\|_{L^p(\bb{R}^n)}. \label{eq:Bernstein}
\end{equation}
\end{thm}

The proof of this version of Bernstein's inequality is a simple adaptation of the proof of Bernstein's inequality for discs: one produces a family of invertible affine transformations $\{T_{\xi,\delta}:\: \xi\in\bb{R}^{n-1},\;\delta>0\}$ such that each $T_{\xi,\delta}$ maps $P_{\xi,\delta}$ to the cube $[-1,1]^n$. Fixing $\varphi\in\mathcal{S}(\bb{R}^n)$ such that $\hat{\varphi}\equiv1$ on $[-1,1]^n$, we see that if $g\in\mathcal{S}(\bb{R}^n)$ has Fourier support in $P_{\xi,\delta}$, then $g=g*(\hat{\varphi}\circ T_{\xi,\delta})\,\widecheck{}\,$, from which Young's convolution inequality implies the result.

We are now ready to give a new proof of the Tomas restriction theorem based on decoupling.

\begin{proof}[Proof of the Tomas restriction theorem.]
Let $R\geq 1$, and suppose $g\in\mathcal{S}(\bb{R}^n)$ has Fourier support in $\mathcal{N}_{R^{-1}}$. Let $\Sigma_R=R^{-1/2}\bb{Z}^{n-1}\cap[-1,1]^{n-1}$, and note that $\mathcal{N}_{R^{-1}}$ is covered by the collection $\mathcal{P}_R=\{P_{\xi,R^{-1/2}} :\; \xi\in\Sigma_R\}$ provided the constant $C>0$ in the definition of the regions $P_{\xi,\delta}$ was chosen large enough to begin with. Let $(\eta_{P})_{P\in\mathcal{P}_R}$ be a partition of unity subordinate to this cover, and for each $P\in\mathcal{P}_{R}$, let $g_{P}=(\hat{g}\eta_{P})\,\widecheck{}\,$. Each $g_{P}$ is Schwartz as the inverse Fourier transform of a Schwartz function, and by Fourier inversion, $\widehat{g_{P}}=\hat{g}\eta_{P}$ is supported in $P$. It follows from the property $\sum_{P\in\mathcal{P}_R}\eta_{P}=1$ that $\sum_{P\in\mathcal{P}_R}\widehat{g_{P}}=\hat{g}$ and hence, upon Fourier inversion, we have $\sum_{P\in\mathcal{P}_R}g_{P}=g$. The $l^2$ decoupling theorem therefore gives
\begin{equation}
    \|g\|_{L^{\frac{2(n+1)}{n-1}}(\bb{R}^n)} = \Big\|\sum_{P\in\mathcal{P}_R}g_{P}\Big\|_{L^{\frac{2(n+1)}{n-1}}(\bb{R}^n)} \lesssim_{\epsilon} R^{\epsilon}\Big(\sum_{P\in\mathcal{P}_R}\|g_{P}\|_{L^{\frac{2(n+1)}{n-1}}(\bb{R}^n)}^2\Big)^{1/2} \label{eq:localtomas.1}
\end{equation}
for all $\epsilon>0$. Since $\widehat{g_{P}}$ is supported in $P$, Theorem \ref{thm:Bernstein} and Plancherel's theorem give
\begin{align*}
    \|g_{P}\|_{L^{\frac{2(n+1)}{n-1}}(\bb{R}^n)} & \lesssim |P|^{\frac{1}{2}-\frac{n-1}{2(n+1)}}\|g_{P}\|_{L^2(\bb{R}^n)}\\
    & \lesssim (R^{-(n+1)/2})^{\frac{1}{2}-\frac{n-1}{2(n+1)}}\|\widehat{g_{P}}\|_{L^2(\bb{R}^n)}\\
    & = R^{-1/2}\|\widehat{g_{P}}\|_{L^2(\bb{R}^n)},
\end{align*}
from which \eqref{eq:localtomas.1} gives
\begin{align}
    \|g\|_{L^{\frac{2(n+1)}{n-1}}(\bb{R}^n)} & \lesssim_{\epsilon} R^{\epsilon-1/2}\Big(\sum_{P\in\mathcal{P}_R}\|\widehat{g_{P}}\|_{L^2(\bb{R}^n)}^2\Big)^{1/2} \label{eq:localtomas.2}
\end{align}
for all $\epsilon>0$. Now, each point in $\bb{R}^n$ lies in at most $O(1)$ of the regions $P$, and it follows that $\sum_{P\in\mathcal{P}_R}|\eta_{P}|^2=O(1)$. We therefore have
\begin{align*}
    \sum_{P\in\mathcal{P}_R}\|\widehat{g_{P}}\|_{L^2(\bb{R}^n)}^2 = \sum_{P\in\mathcal{P}_R}\int_{\bb{R}^n}|\hat{g}\eta_{P}|^2\,dx & = \int_{\bb{R}^n}|\hat{g}|^2\Big(\sum_{P\in\mathcal{P}_R}|\eta_{P}|^2\Big)\,dx\\
    & \lesssim \|\hat{g}\|_{L^2(\bb{R}^n)}^2,
\end{align*}
which together with \eqref{eq:localtomas.2} gives
\begin{align*}
    \|g\|_{L^{\frac{2(n+1)}{n-1}}(\bb{R}^n)} & \lesssim_{\epsilon} R^{\epsilon-1/2} \|\hat{g}\|_{L^2(\bb{R}^n)}.
\end{align*}
In particular, inequality \eqref{eq:suff} holds, from which the Tomas restriction theorem follows.
\end{proof}

\section{\texorpdfstring{$\epsilon$}{Epsilon}-removal}

We now outline a new approach to proving Theorem \ref{thm:epsilonremoval}. In fact, we will prove the following result for a more general range of exponents, yielding Theorem \ref{thm:epsilonremoval} as a simple corollary:

\begin{thm}\label{thm:epsilonremoval.1}
Given $2\leq p \leq q\leq\infty$ and $\epsilon>0$ sufficiently small, $\Rstar{}{p}{q}{\,;\,\epsilon}$ implies $\Rstar{}{p}{r}{}$ whenever
$$\frac{1}{r} < \frac{1}{q_0} : =  \frac{1}{q}-\frac{4\log 2}{q\log(1/\epsilon q)}.$$
\end{thm}

Observe that if the weak-type bound
\begin{equation}
    \|Ef\|_{L^{q_0,\infty}(\bb{R}^n)} \lesssim \|f\|_{L^{p}(P^{n-1},\,d\sigma)},\quad\forall f\in L^p(P^{n-1},d\sigma) \label{eq:weaktypereduction}
\end{equation}
is known, then Theorem \ref{thm:epsilonremoval.1} follows by interpolation with the known estimate $\Rstar{}{p}{\infty}{}$. Our goal is therefore to prove the weak-type estimate \eqref{eq:weaktypereduction}. It is via this reduction that we obtain a more natural proof of Theorem \ref{thm:epsilonremoval.1} than that given for the original $\epsilon$-removal result of Tao (Theorem $1.2$, \cite{TaoBR}). The argument of Tao was phrased in terms of the restriction formulation, in which case equation \eqref{eq:weaktypereduction} was replaced by a dual Lorentz space estimate of the form
\begin{equation}
    \|\hat{g}|_{P^{n-1}}\|_{L^{p'}(P^{n-1},\,d\sigma)} \lesssim
    \|g\|_{L^{{q_0}',1}(\bb{R}^n)},\quad\forall g\in\mathcal{S}(\bb{R}^n).
\end{equation}
To prove this estimate, Tao sketched some further reductions of a highly technical nature which are quite laborious when fully elaborated. By adhering to the extension formulation, we obtain a more natural proof which we can be more easily elaborated.

Though we adopt a different formulation to Tao, our proof will use the same key components. In particular, we borrow the following definition and key lemmas from \cite{TaoBR} (with minor modifications):

\begin{definition}
A collection $\{B(x_i,R)\}_{i=1}^{N}$ of $R$-balls is said to be \textit{sparse} if $|x_j-x_i|\gtrsim (RN)^{\frac{2}{n-1}}$ for all $i\neq j$.
\end{definition}

The first key lemma allows us to bootstrap the local extension estimate $\|E f\|_{L^{q}(B(0,R))}\lesssim R^{\epsilon}\|f\|_{L^{p}(P^{n-1},\,d\sigma)}$
to an analogous estimate for multiple $R$-balls $\bigcup_{i}B(x_i,R)$, provided the collection of balls is sparse.

\begin{lem}[Lemma $3.2$, \cite{TaoBR}] \label{lem:epsilonremoval.1}
Let $2\leq p\leq q\leq\infty$ and $\epsilon>0$, and suppose that $\Rstar{}{p}{q}{\,;\,\epsilon}$ holds. Then for all $R\geq 1$ and all $f\in L^p(P^{n-1},d\sigma)$, we have
\begin{equation}
    \|Ef\|_{L^{q}(\bigcup_{i=1}^{N}B(x_i,R))} \lesssim R^{\epsilon} \|f\|_{L^p(P^{n-1},\,d\sigma)}
\end{equation}
whenever $\{B(x_i,R)\}_{i=1}^{N}$ is a sparse collection of $R$-balls.
\end{lem}

The second key lemma allows us to cover a set which is a union of unit cubes by a reasonably small number of sparse collections of balls of sufficiently small radius.

\begin{lem}[Lemma $3.3$, \cite{TaoBR}] \label{lem:epsilonremoval.2}
Let $A$ be a union of unit cubes. For any $N\geq 1$, $A$ may be covered by $O(N|A|^{1/N})$ sparse collections of balls of radius $O(|A|^{2^N})$.
\end{lem}

In the forthcoming proof, we will have frequent need to consider various superlevel sets, so we introduce some notation here for convenience:

\begin{notn}
Let $f:\bb{R}^n\to\bb{C}$ be measurable. Given $t>0$, let $\eta_f(t)$ denote the superlevel set
\begin{equation}
    \eta_f(t) = \{x\in \bb{R}^n : |f(x)|\geq t\}.
\end{equation}
\end{notn}

Given $r>0$, we will also denote by $Q_r$ the cube $[-r,r]^n$.

\begin{proof}[Proof of Theorem \ref{thm:epsilonremoval.1}.]
Let $2\leq p\leq q\leq\infty$, $\epsilon>0$, and suppose $\Rstar{}{p}{q}{\,;\,\epsilon}$ holds. Recall that it suffices to prove the estimate \eqref{eq:weaktypereduction}.

Let $f\in L^p(P^{n-1},d\sigma)$, and assume without loss of generality (by scale invariance) that $\|f\|_{L^p(P^{n-1},\,d\sigma)}=1$. We begin by bounding the left-hand side of \eqref{eq:weaktypereduction} by the $L^{q_0,\infty}(\bb{R}^n)$ norm of the average $|Ef|*\chi_{Q_{1/4}}$, motivated by the heuristic that the superlevel sets of $|Ef|*\chi_{Q_{1/4}}$ resemble unions of unit cubes, so Lemma \ref{lem:epsilonremoval.2} will become applicable.

By Fubini's theorem, we may compute
\begin{equation}
    \widehat{\chi_{Q_{1/4}}}(\xi) = \prod_{i=1}^{n}\frac{\sin(\pi\xi_i/2)}{\pi \xi_i},
\end{equation}
so $\widehat{\chi_{Q_{1/4}}}\sim 1$ on $Q_{3/2}$. Let $\varphi\in C^{\infty}_c(Q_{3/2})$ be a bump function with $\varphi\equiv1$ on $Q_1$, and define
\begin{equation}
    \psi : = \begin{cases}
    \varphi/\widehat{\chi_{Q_{1/4}}} & \text{on }Q_{3/2}\,;\\
    0 & \text{on }\bb{R}^n\setminus Q_{3/2}.
    \end{cases}
\end{equation}
Since $\varphi\equiv 1$ on $P^{n-1}\subset Q_1$, Fubini's theorem gives $Ef=Ef*\check{\varphi}$. But clearly, $\varphi=\widehat{\chi_{Q_{1/4}}}\psi$, from which we see by Fourier inversion that $\check{\varphi}=\chi_{Q_{1/4}}*\check{\psi}$, hence $Ef=Ef*\chi_{Q_{1/4}}*\check{\psi}$. Young's convolution inequality for weak $L^p$ spaces therefore gives $\|Ef\|_{L^{q_0,\infty}(\bb{R}^n)} \lesssim \||Ef|*\chi_{Q_{1/4}}\|_{L^{q_0,\infty}(\bb{R}^n)}$, so by our normalisation, it suffices to prove $\||Ef|*\chi_{Q_{1/4}}\|_{L^{q_0,\infty}(\bb{R}^n)} \lesssim 1$. That is, it suffices to prove
\begin{equation}
    |\eta_{|Ef|*\chi_{Q_{1/4}}}(t)|^{1/q_0} \lesssim \frac{1}{t},\quad\forall t>0. \label{eq:weaktypereduction.1}
\end{equation}
We first note that \eqref{eq:weaktypereduction.1} is clear for all $t$ larger than a particular threshold independent of $f$. Indeed, Young's convolution inequality followed by the known estimate $\Rstar{}{p}{\infty}{}$ and our normalisation $\|f\|_{L^p(P^{n-1},\,d\sigma)} = 1$ gives
\begin{equation}
    \||Ef|*\chi_{Q_{1/4}}\|_{L^{\infty}(\bb{R}^n)} \leq \|Ef\|_{L^{\infty}(\bb{R}^n)}\|\chi_{Q_{1/4}}\|_{L^1(\bb{R}^n)} \lesssim 1.
\end{equation}
It follows that $\||Ef|*\chi_{Q_{1/4}}\|_{L^{\infty}(\bb{R}^n)}\leq K$ for some constant $K$ independent of $f$, from which we see that $|\eta_{|Ef|*\chi_{Q_{1/4}}}(t)|=0$ for $t>K$. Equation \eqref{eq:weaktypereduction.1} therefore clearly holds for $t>K$, so we may assume without loss of generality that $0<t\leq K$.

Given $t>0$, let $A_t$ be the union of all unit cubes of the form $z+Q_{1/2}$ for some $z\in\bb{Z}^n$ which have nonempty intersection with the superlevel set $\eta_{|Ef|*\chi_{Q_{1/4}}}(t)$. Clearly, $\eta_{|Ef|*\chi_{Q_{1/4}}}(t)\subset A_t$, hence $|\eta_{|Ef|*\chi_{Q_{1/4}}}(t)|^{1/q_0}\leq|A_t|^{1/q_0}$; it therefore suffices to prove
\begin{equation}
    |A_t|^{1/q_0} \lesssim \frac{1}{t},\quad\forall 0<t\leq K.  \label{eq:weaktypereduction.2}
\end{equation}
Since $A_t$ is a union of unit cubes, Lemma \ref{lem:epsilonremoval.2} allows us to cover $A_t$ by sparse collections $\mathcal{C}_1,\cdots,\mathcal{C}_m$ of balls of radius $O(|A_t|^{2^N})$ for any $N$, where $m=O(N|A_t|^{1/N})$. Since these collections cover $A_t$, we have
\begin{equation}
    A_t = \bigcup_{i=1}^{m}\Big(A_t\cap\bigcup_{B\in\mathcal{C}_i}B\Big),
\end{equation}
hence,
\begin{equation}
    |A_t| \leq \sum_{i=1}^{m}\Big|A_t\cap\bigcup_{B\in\mathcal{C}_i}B\Big|. \label{eq:epsilonremoval.1}
\end{equation}
If $x\in A_t$ then by definition, there exists $z\in\bb{Z}^n$ and $y\in\bb{R}^n$ such that $x,y\in z+Q_{1/2}$ and $|Ef|*\chi_{Q_{1/4}}(y)\geq t$. It follows that $|x-y|\leq\sqrt{n}$ and hence, $|Ef|*\chi_{Q_{1/4+\sqrt{n}}}(x)\geq|Ef|*\chi_{Q_{1/4}}(y)\geq t$. We therefore have $A_t\subset\eta_{|Ef|*\chi_{Q_{1/4+\sqrt{n}}}}(t)$, and in particular, equation \eqref{eq:epsilonremoval.1} gives
\begin{align}
    |A_t| \leq \sum_{i=1}^{m}\Big|A_t\cap\bigcup_{B\in\mathcal{C}_i}B\Big| & \leq \sum_{i=1}^{m}\Big|\eta_{|Ef|*\chi_{Q_{1/4+\sqrt{n}}}}(t)\cap\bigcup_{B\in\mathcal{C}_i}B\Big|\\
    & = \sum_{i=1}^{m}|\eta_{(|Ef|*\chi_{Q_{1/4+\sqrt{n}}})\chi_{\bigcup_{B\in\mathcal{C}_i}B}}(t)|\\
    & \leq \frac{1}{t^{q}}\sum_{i=1}^{m}\||Ef|*\chi_{Q_{1/4+\sqrt{n}}}\|_{L^{q}(\bigcup_{B\in\mathcal{C}_i}B)}^{q}, \label{eq:epsilonremoval.2}
\end{align}
where the last line follows by Chebyshev's inequality. But for all $x\in\bigcup_{B\in\mathcal{C}_i}B$, we have
\begin{equation}
    |Ef|*\chi_{Q_{1/4+\sqrt{n}}}(x) = (|Ef|\chi_{\bigcup_{B\in\mathcal{C}_i}(B+Q_{1/4+\sqrt{n}})})*\chi_{Q_{1/4+\sqrt{n}}}(x).
\end{equation}
It follows that
\begin{align}
    \||Ef|*\chi_{Q_{1/4+\sqrt{n}}}\|_{L^{q}(\bigcup_{B\in\mathcal{C}_i}B)}^{q} & \leq \|(|Ef|\chi_{\bigcup_{B\in\mathcal{C}_i}(B+Q_{1/4+\sqrt{n}})})*\chi_{Q_{1/4+\sqrt{n}}}\|_{L^q(\bb{R}^n)}^q\\
    & \lesssim \|Ef\|_{L^{q}(\bigcup_{B\in\mathcal{C}_i}(B+Q_{1/4+\sqrt{n}}))}^{q}, \label{eq:epsilonremoval.3}
\end{align}
where the last line follows by Young's convolution inequality.
But clearly, $B(x,R)+Q_{1/4+\sqrt{n}}\subset B(x,R+2n)$ for any ball $B(x,R)\subset\bb{R}^n$. Letting $\tilde{\mathcal{C}}_i$ denote the collection of balls obtained by enlarging the radius of each ball in $\mathcal{C}_i$ by $2n$, it follows that
\begin{equation}
    \|Ef\|_{L^{q}(\bigcup_{B\in\mathcal{C}_i}(B+Q_{1/4+\sqrt{n}}))}^{q} \leq \|Ef\|_{L^{q}(\bigcup_{B\in\tilde{\mathcal{C}}_i}B)}^{q},
\end{equation}
which together with \eqref{eq:epsilonremoval.2} and \eqref{eq:epsilonremoval.3} implies
\begin{equation}
    |A_t| \lesssim \frac{1}{t^{q}}\sum_{i=1}^{m}\|Ef\|_{L^{q}(\bigcup_{B\in\tilde{\mathcal{C}}_i}B)}^{q}. \label{eq:epsilonremoval.4}
\end{equation}
Noting that each of the collections $\tilde{\mathcal{C}}_i$ is also sparse and is comprised of balls of radius $\lesssim |A_t|^{2^N}+2n$, Lemma \ref{lem:epsilonremoval.1} and our normalisation give
\begin{equation}
    \|Ef\|_{L^{q}(\bigcup_{B\in\tilde{\mathcal{C}}_i}B)}^{q} \lesssim (|A_t|^{2^N}+2n)^{\epsilon q}. \label{eq:epsilonremoval.5}
\end{equation}
Now, if $|A_t|<2n$, we get $|A_t|^{1/q_0} \lesssim 1 \lesssim 1/t$ (recalling that $0<t\leq K$), which is the conclusion \eqref{eq:weaktypereduction.2}. On the other hand, if $2n\leq |A_t|$, then \eqref{eq:epsilonremoval.5} gives
\begin{equation}
    \|Ef\|_{L^{q}(\bigcup_{B\in\tilde{\mathcal{C}}_i}B)}^{q} \lesssim |A_t|^{\epsilon q 2^N}. \label{eq:epsilonremoval.6}
\end{equation}
Combining \eqref{eq:epsilonremoval.4} and \eqref{eq:epsilonremoval.6} and recalling that $m=O(N|A_t|^{1/N})$, we find that
\begin{equation}
    |A_t| \lesssim \frac{1}{t^{q}}N|A_t|^{\epsilon q2^N+1/N}. \label{eq:epsilonremoval.7}
\end{equation}
Setting $N=\frac{\log(1/\epsilon q)}{2\log 2}$, we get $\epsilon q2^{2N}=1$, hence $\epsilon q2^N=2^{-N}\leq 1/N$. Noting that $|A_t|\geq 1$, \eqref{eq:epsilonremoval.7} then gives $|A_t| \lesssim \frac{1}{t^{q}}N|A_t|^{2/N}$ for our particular choice of $N$. Hence, if $|A_t|$ is finite, we may rearrange to conclude that $|A_t|^{1/q-2/Nq} \lesssim \frac{N^{1/q}}{t}$. Substituting our expression for $N$, this is equivalent to
\begin{equation}
    |A_t|^{\frac{1}{q}-\frac{4\log 2}{q\log(1/\epsilon q)}} \lesssim \frac{1}{t}, \label{eq:epsilonremoval.8}
\end{equation}
which is the conclusion \eqref{eq:weaktypereduction.2}. To see that $|A_t|$ is indeed finite, we note that by the asymptotics of the Fourier transform of a surface-supported measure and a density argument, $Ef(x)$ decays to $0$ as $|x|\to\infty$. It follows that the superlevel set $\eta_{|Ef|*\chi_{Q_{1/4}}}(t)$ is bounded, from which it is clear that $A_t$ is also bounded, and therefore has finite measure.
Our manipulations leading to \eqref{eq:epsilonremoval.8} are therefore justified, and we are done.
\end{proof}

Observe that as $\epsilon\to 0$, we have $\frac{1}{q}-\frac{4\log 2}{q\log(1/\epsilon q)}\to\frac{1}{q}$. When combined with Theorem \ref{thm:epsilonremoval.1}, this observation yields the following simple corollary of which Theorem \ref{thm:epsilonremoval} is a special case:

\begin{cor}
Given $2\leq p\leq q\leq\infty$, the local extension estimates $\Rstar{}{p}{q}{\,;\,\epsilon}$ for all $\epsilon>0$ imply $\Rstar{}{p}{r}{}$ for all $r>q$.
\end{cor}

\bibliographystyle{alpha}
\bibliography{main}

\end{document}